\documentclass[10pt,oneside,a4paper]{article}

\usepackage{amssymb}
\usepackage{amsmath}
\usepackage{amsthm}
\usepackage{color}
\usepackage{graphicx}
\usepackage{wrapfig}
\usepackage{float}
\usepackage{subfigure}
\usepackage{rotating}
\usepackage{cite}
\usepackage{caption}
\usepackage{hyperref}
\usepackage{enumitem}

\providecommand{\e}[1]{\ensuremath{\cdot 10^{#1}}}
\def\bu{{\bar{u}}}

\newtheorem{thm}{Theorem}[section]

\newtheorem{lem}[thm]{Lemma}

\theoremstyle{remark}
\newtheorem{rem}[thm]{Remark}

\theoremstyle{definition}

\newcounter{algor}
\newtheorem{algorithm}[algor]{Algorithm}

\def\Id{{\rm Id}}

\newcommand{\comment}[1]{}
\let\oldmarginpar\marginpar
\renewcommand\marginpar[1]{\-\oldmarginpar[\raggedleft\footnotesize #1]
{\raggedright\footnotesize #1}}

\author{
Jordi-Llu\'is Figueras
\thanks
{
Department of Mathematics, Uppsala University, 
Box 480, 75106 Uppsala (Sweden). {\tt figueras@math.uu.se}.
}
\and
Rafael de la Llave 
\thanks
{
School of Mathematics, Georgia Institute of Technology, 
686 Cherry St NW, Atlanta, GA 30332, United States.
{\tt rafael.delallave@math.gatech.edu}. R.L. 
is supported in part by NSF DMS 1500943
}
}

\title{Numerical Computations and Computer Assisted Proofs
of Periodic Orbits of the Kuramoto-Sivashinsky Equation}

\begin{document}
	
\maketitle

\begin{abstract}
We present numerical results and computer assisted proofs of the existence of
periodic orbits for the Kuramoto-Sivashinky equation. These two results are
based on writing down the existence of periodic orbits as zeros of functionals.
This leads to the use of Newton's algorithm for the numerical computation of
the solutions and, with some a posteriori analysis in combination with rigorous
interval arithmetic, to the rigorous verification of the existence of
solutions.  This  is a particular case of the methodology developed in
\cite{ks_theoretical} for several types of orbits.  An  independent
implementation, covering overlapping but different ground, using different
functional setups, appears in \cite{ks_jp_marcio}. 
\end{abstract}

\begin{center}
{\bf \small Keywords} \\ \vspace{.05cm}
{ \small Evolution equation $\cdot$ Periodic Orbits $\cdot$ Contraction mapping \\
$\cdot$ Rigorous Computations $\cdot$ Interval Analysis}
\end{center}

\begin{center}
{\bf \small Mathematics Subject Classification (2010)} \\ \vspace{.05cm}
{ \small 35B32 $\cdot$ 35R20 $\cdot$ 47J15 $\cdot$  65G40 $\cdot$ 65H20 }
\end{center}

\section{Introduction.}\label{section: introduction}

In \cite{ks_theoretical} one can find a theoretical framework for the
computation and rigorous computer assisted verification of invariant objects
(fixed points, travelling waves, periodic orbits, attached invariant manifolds)
of semilinear parabolic equations of the form
\[
\partial_t u +Lu+N(u) = 0, 
\]
where $L$ is a linear operator and $N$ is nonlinear. The two operators $L$ and
$N$ are possibly unbounded but satisfy that $L^{-1}N$ is continuous.  The
methodolody of \cite{ks_theoretical}< is based on writing down an invariance
equation for these objects in suitable Banach spaces. One remarkable aspect of
this methodology is that if one applies a posteriori constructive methods one
can obtain computer assisted proofs and validity theorems.

In this paper we apply this methodology for the numerical computation and
a posteriori rigorous verification of the existence of periodic orbits in a
concrete example: the Kuramoto-Sivashinsky equation. This equation is the
parabolic semilinear partial differential equation

\begin{equation}\label{eq: ks equation}
\partial_t u + \left(\nu \partial_x^4+\partial_x^2\right)
u+\frac12 \partial_x\left(u^2\right) = 0,
\end{equation}
where $\nu > 0$ and $u\colon\mathbb{R}\times\mathbb{T}\rightarrow \mathbb{R}$.
($\mathbb{T} := \mathbb{R}/(2\pi \mathbb{Z}) $).  We restrict our study to  the
space of periodic odd functions, $u(t, x) = -u(t, -x)$,
\begin{equation*}
u(t, x) = \sum_{k=1}^\infty a_k(t)\sin(k x). 
\end{equation*}

The PDE \eqref{eq: ks equation} is used in the study of several physical
systems. For example, instabilities of dissipative trapped ion modes in plasmas
\cite{laqueyetaltri1975, Cohenetaltri1976}, instabilities in laminar flame
fronts \cite{Sivashinksy77} and  phase dynamics in reaction-diffusion systems
\cite{Kuramoto76}.

The Kuramoto-Sivashinky equation has been extensively studied both
theoretically and numerically \cite{Armbruster, Colletattracting,
Colletanalyticity, Ilyashenko, Nicolaenkoetaltri}.  It satisfies that its flow
is well-posed forward in time in Sobolev, $L^2$ and analytic spaces. In fact,
it is smoothing: For positive values of $t$ the solutions with $L^2$ initial
data are analytic in the space variable $x$. The phase portrait depends on the
value of the parameter $\nu$: The zero solution $ u(t, x) = 0 $ is a fixed
point with a finite dimensional unstable manifold. Its dimension is the number
of solutions of the integer inequality $k^2-\nu k^4 > 0$, $k > 0$.  For $\nu >
1$ the zero solution is a global attractor of the system.  For every $\nu > 0$,
the system has a global attractor. This attractor has finite dimension (it is
confined inside an inertial manifold), \cite{Jolly_Kevrekidis_Titi_90,
Foias_Nicolaenko_Sell_Teman_88, EdenFNT94, Temam97, Chueshov02, CFNT_book,
Robinson_book}. Finally, it has plenty of periodic orbits
\cite{LanCvitanovic2008, CvitanovicDavidchackSiminos2010}. For example, it is
known empirically that there are period doubling cascades
\cite{Papageorgiou_Smyrlis_90, Papageorgiou_Smyrlis_91} satisfying the same
universality properties than  in \cite{Feigenbaum78,TresserC78}. See Section
\ref{section: num explor} for a numerical exploration of the phase portrait of
the Kuramoto-Sivashinsky equation \eqref{eq: ks equation}.

In the literature several ways have been proposed for computing periodic orbits
of the Kuramoto-Sivashinsky equation: If the periodic orbit is attracting, one
can use an ODE solver for computing the evolution of the system using Galerkin
projections. Accordingly, starting at an initial point in the basin of
attraction and integrating forward in time one gets close to the periodic
orbit. If the periodic orbit is unstable, another classical technique is to
compute them as fixed points of some Poincar\'e map of the system. Another
approach, \emph{the Descent method}, is presented in
\cite{LanChandreCvitanovic2006, LanCvitanovic2008}. This is a method that,
given an initial guess of the periodic orbit, it evolves it under a variational
method minimizing the local errors of the initial guess. 

In this paper we implement another method based on solving, using
\emph{Newton's method}, a functional equation that periodic orbits satisfy. The
unknowns are the frequency and the parameterization of the periodic orbit.
This methodology permits us to write down a posteriori theorems that, with the
help of rigorous computer assisted verifications, lead us to the rigorous
verification of these periodic orbits by estimating all 
the sources of error (truncation, roundoff). In this paper, we 
carry out this estimates, so that the results we present are 
rigorous theorems on existence of periodic orbits. 

The Newton method, of course, has the shortcoming that it depends on having a
close initial guess; the descent method in practice has a larger domain of 
convergence. 
On the other hand, the Newton
method produces solutions to machine epsilon precision $\varepsilon_M$, whereas
the descent method, being a variational method, cannot get beyond
$\sqrt\varepsilon_M$ and, moreover, slows down near the solution and 
may have problems with stiffness.  Other
variational algorithms (e.g. conjugate gradient, Powell \cite{Brent73}
 or Sobolev gradients
\cite{Neuberger10} ) could be faster and less 
sensitive to stiffness  even if
limited to $\sqrt\varepsilon_M$ precision. Of course, one can combine both
methods and obtain convergent methods up to machine epsilon: Gradient
like methods at the beginning but switching to fast Newton's method for the end
game.

The goal of this paper is not only to obtain numerical computations 
but also to estimate all the sources of error and to obtain computer
assisted proofs of the existence of the numerical orbits obtained 
and some of their properties. 

There has been other computer assisted proofs of invariant objects of the
Kuramoto-Sivashinky equation.  In \cite{ArioliKoch2, Piotr1, Piotr2} the
authors prove the existence of stationary solutions and their bifurcation
diagrams, and in \cite{ArioliKoch1, Piotr3} they prove the existence of
periodic orbits. The proof is done there by  combining rigorous propagation of
the (semi-) flow defined by the PDE and a fixed point theorem in a suitable
Poincar\'e section. 

In this paper, 
the flow property is not used: the existence  of periodic
orbits is reduced to a smooth functional
defined in a Banach space. This methodology could be used for the proof of the
existence of periodic orbits in other type of PDEs 
See \cite{ks_theoretical} for a systematic study. 
 Remarkably, in \cite{CGL_Boussinesq}
the methodology has been extended to validate numerical 
periodic solutions of 
\begin{equation}\label{boussinesq}
\partial_{tt} u =   \left(\nu \partial_x^4+\partial_x^2\right)
u+\frac12 \partial_{xx}\left(u^2\right) \quad \mu > 0 
\end{equation}
with periodic boundary conditions. It is to 
be noted that \eqref{boussinesq} which does  not define a flow, so 
that the methods of finding periodic solutions based on propagating, 
cannot get started. 
Rigorous a-posteriori theorems of existence of quasi-periodic solutions in \eqref{boussinesq}
are in 
\cite{LlaveS16}.

An independent implementation of the  methodology in \cite{ks_theoretical} to 
the Kuramoto-Shivashisly is in \cite{ks_jp_marcio}.  The papers
\cite{ks_jp_marcio} and this one, even if they share 
a common philosophy (explained in \cite{ks_theoretical} )
 differ in several aspects: the spaces of functions considered, using different 
results to control the errors of numerical. The paper \cite{ks_jp_marcio} 
also considered branching of the continuations. 

\paragraph{Organization of the paper}

In Section \ref{section: invariance equation} we present the invariance
equation for the periodic orbits. Then, in Section \ref{section: Newton scheme}
we develop the numerical scheme for the computation of these orbits.  The
methodology for the validation of the the periodic orbits is presented in
Sections \ref{section: a posteriori theorems} and \ref{section: implementation
theorem}.  In Section \ref{section: a posteriori theorems} we present a theorem
that leads to the validation of the periodic orbits, and in Section
\ref{section: implementation theorem} we deduce a rigorous numerical scheme for
the verification of the existence of periodic orbits.  Later, in Section
\ref{section: numerical examples}, several examples of the numerical and the
rigorous schemes are described.  In Appendix \ref{section: appendix} we define
the functional spaces and the properties used during the computer assisted
proofs. In Appendix \ref{section: multiplication} we present a fast algorithm
due to \cite{Rump_matrices_0} for multiplying high dimensional interval
matrices. This algorithm is used for the application of the rigorous 
numerical scheme.

\subsection{Non-rigorous exploration: Period-doubling cascades}
\label{section: num explor}

In this heuristic chapter, 
we  use the remarkable fact the Kuramoto-Sivashinsky equation \eqref{eq: ks
equation} has period-doubling cascades as a source for periodic orbits 
that later we will validate rigorously.  

Computing nonrigorously attracting periodic orbits and period-doubling cascades
is easy: it just requires to integrate forward in time a random (but
well-selected) initial condition until it gets close to the attracting orbit.
Let's give a brief description of the method.  More details can be found in
\cite{Papageorgiou_Smyrlis_91, LanCvitanovic2008}.

Given an initial condition 
\begin{equation*}
u(0, x) = \sum_{k=1}^\infty a_k(0)\sin(kx),
\end{equation*}
it is easy to see that its Fourier coefficients evolve via the (infinite
dimensional) system of 
differential equations  
\begin{equation}
\label{eq: infinite ode}
\dot{a}_k = \left(k^2-\nu k^4\right)a_k+\frac{k}{2}
\left(\sum_{l=1}^\infty a_{k+l}a_k-\frac12\sum_{l+m=k}a_la_m\right)
.
\end{equation}
After truncating the system \eqref{eq: infinite ode}, we get a finite
dimensional ODE: Since it is rather stiff, we should be careful with the ODE
solver we choose.  Numerical tests show that Runge-Kutta 4-5 is enough for our
purposes. Hence, after fixing a value of the parameter $\nu$ and starting with
the initial point $u(0, x)=\sin(x)$, we integrate it forwards in time and,
after a transient time, obtain a good approximation of the periodic orbit.  In
b) to d) in Figure \ref{figure: sections po cascade} we can see the $a_1-a_2$
coordinates of some periodic orbits for different values of the parameter
$\nu$. The period-doubling cascade can be visualized by plotting the local
minima in time of the $L^2-$energy
\begin{equation*}
\text{Energy}(t) = \sqrt{\sum_{k=1}^\infty a_k(t)^2},
\end{equation*}
along the periodic orbit, see a) in Figure \ref{figure: sections po
cascade}.

\begin{figure}
\begin{center}
\resizebox{120mm}{!}{
\begin{tabular}{cc}
\resizebox{60mm}{!}{\includegraphics[angle=270]{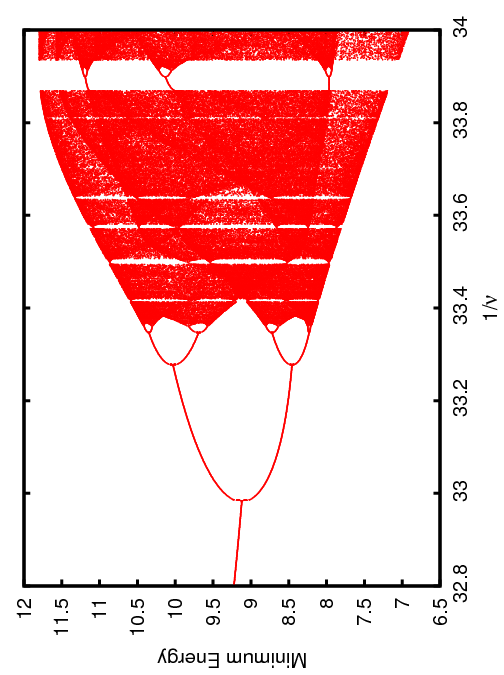}}&
\resizebox{60mm}{!}
{\includegraphics[type=png,ext=.png,read=.png,angle=270]
{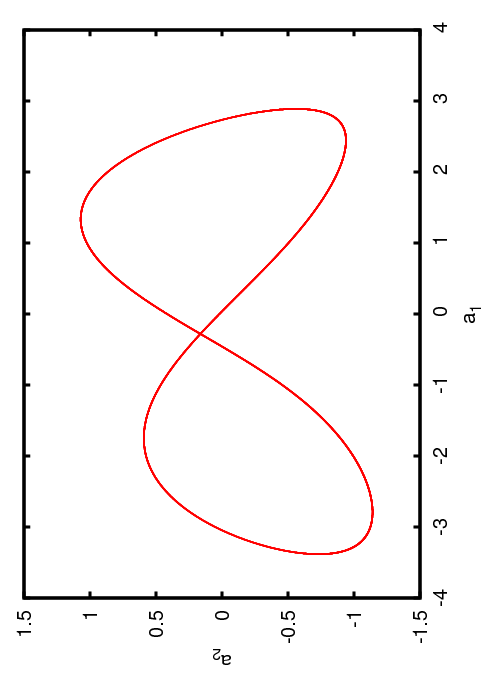}} \\
a) Period doubling cascade &
b) {$1/\nu = 33.2701$} \\
\resizebox{60mm}{!}{\includegraphics[type=png,ext=.png,read=.png,angle=270]
{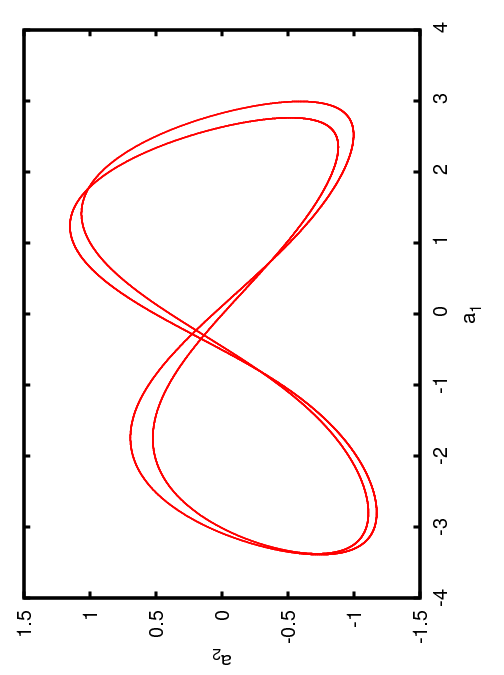}} &
\resizebox{60mm}{!}{\includegraphics[type=png,ext=.png,read=.png,angle=270]
{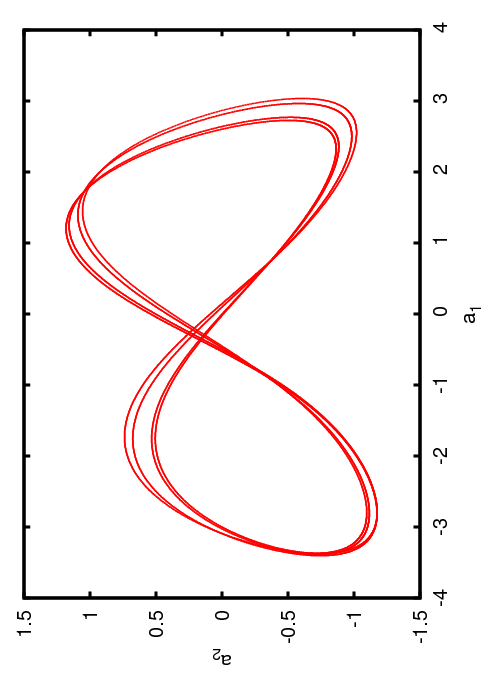}} \\
c) {$1/\nu = 33.3353$} &
d) {$1/\nu = 33.3569$} \\
\end{tabular}
} \caption{Figure a) shows a period-doubling bifurcation cascade. Figures b) to
d) show the projection on the $a_1-a_2$ coordinates of some periodic orbits of
the first three period-doublings.}\label{figure: sections po cascade}
\end{center}
\end{figure}

\begin{rem} 
The period doubling cascades described above  have, to the limit of numerical
precision, the same quantitative properties than the one-dimensional ones found
in \cite{Feigenbaum78,TresserC78} even if the K-S, in principle, is an infinite
dimensional dynamical system. Nevertheless, since the system admits an inertial
manifold it is plausible that the arguments of \cite{ColletEK81} apply.
\end{rem}

The papers \cite{LanChandreCvitanovic2006,
LanCvitanovic2008,CvitanovicDavidchackSiminos2010} present a very remarkable
explicit surface of section which allows to reduce approximately to a one
dimensional map. Other sources of periodic orbits can be found as a byproduct
of computations of inertial manifolds \cite{Jolly_Kevrekidis_Titi_90,
GarciaNT98, NovoTW01, JollyRT00, DCCST2016}.  In this paper we will not use
these methods, but the periodic orbits found by them could be validated using
the methods here.

\section{Derivation of the invariance equation for the 
periodic orbits.}\label{section: invariance equation}

Here we derive a functional equation for the periodic orbits of the
Kuramoto-Sivashinsky equation. This functional equation
 is well suited for applying a
fixed point problem for a well-defined operator. Later on, with this equation,
we develop a numerical scheme for the computation of these orbits and an a
posteriori verification method.

Periodic orbits with period $T$ of the Kuramoto-Sivashinsky equation \eqref{eq:
ks equation} satisfy, under the time rescaling  $\theta= \dfrac{2\pi t}{T}$,
the invariance equation 
\begin{equation}\label{eq: invariance}
f \partial_\theta u + L
u+\frac12 \partial_x\left(u^2\right) = 0,
\end{equation}
where $L = \nu \partial_x^4+\partial_x^2$ and $f = \frac{2\pi}{T}$. A solution
of Equation \eqref{eq: invariance} is represented by a pair $(f, u)$, where $f$
is a real number and $u:\mathbb T^2\rightarrow \mathbb R$, 
$u(\theta,x)$ is odd with respect $x$.

Given an approximate solution $(f_0, u_0)$ of Equation \eqref{eq: invariance},
we look for a correction $(\sigma, \delta )$ of it. This correction satisfies
the equation
\begin{equation}\label{eq: unbounded equation}
f_0\partial_\theta \delta+L \delta+\partial_x\left(u_0\cdot \delta\right)+
\sigma\partial_\theta u_0 = -e-\frac12\partial_x(\delta^2)
-\sigma\partial_\theta\delta,
\end{equation}
where $e = f_0\partial_\theta u_0+L u_0 +\frac12\partial_x(u_0^2)$ is the error
of the approximation $(f_0, u_0)$.

Equation \eqref{eq: unbounded equation} has two problems:

\begin{enumerate}
\item Solutions of Equation \eqref{eq: invariance} are non-unique: If $(\sigma,
\delta(\theta, x))$ is a solution, then so is $(\sigma, \delta(\theta+a, x))$,
$\forall a\in\mathbb{R}.$

\item The linear part of Equation \eqref{eq: unbounded equation} is an
unbounded operator. This leads to numerical instabilites.

\end{enumerate}

To fix the non-uniqueness problem, we impose another equation so that 
Problem \eqref{eq: unbounded equation} has a unique solution. 

To motivate the choice of normalization, we observe that if 
$u(\theta, x)$ is a solution of Equation \eqref{eq: invariance} satisfying  
\begin{equation}
\label{eq:uniqueness1}
\int_{\mathbb{T}^2} u \cdot \partial_\theta u  = C \text{ (Constant)}.
\end{equation}
Then, a translation in the $\theta$ direction -- the source of 
non-uniqueness changes the quantity \eqref{eq:uniqueness1} by 
\begin{equation*}
\int_{\mathbb{T}^2} u(\theta+a, x) 
\cdot \partial_\theta u(\theta, x)  
\simeq\int_{\mathbb{T}^2} (u(\theta, x)+\partial_\theta u(\theta, x) a) 
\cdot \partial_\theta u(\theta, x) 
= C + a\|\partial_\theta u\|^2_{L^2(\mathbb{T}^2)}. 
\end{equation*}
Thus, 
\begin{equation*}
\int_{\mathbb{T}^2} u(\theta+a, x) 
\cdot \partial_\theta u(\theta, x)=
C+a\|\partial_\theta u\|^2_{L^2(\mathbb{T}^2)}+O(a^2).
\end{equation*}
The above calculation can be interpreted geometrically saying that the 
surface in function space given by  \eqref{eq:uniqueness1} is 
transversal to the symmetries of the equation. 

Therefore, we impose local uniqueness for Equation \eqref{eq: unbounded
equation} by requiring that the correction $\delta$ should be
\textit{perpendicular} to the approximate parameterization $u_0$. That is,  
\begin{equation*}
\int_{\mathbb{T}^2} \delta \cdot \partial_\theta u_0  = 0.
\end{equation*}

As we observed before, the linear operator $f_0\partial_\theta+L$ is unbounded,
but we can transform  Equation \eqref{eq: unbounded equation} into 
a smooth equation  by performing algebraic manipulations.   Let
$c\in\mathbb{R}$ be such that $S_c = f_0\partial_\theta+L+c \Id$ is invertible.
Then, we have that $(\sigma, \delta)$ in Equation \eqref{eq: unbounded equation}
satisfies the equation

\begin{equation*}
A
\begin{pmatrix}
\sigma\\
\delta
\end{pmatrix}
=
\tilde e
+
\tilde N
\begin{pmatrix}
\sigma\\
\delta
\end{pmatrix}
,
\end{equation*}
where 
\begin{equation}\label{eq: operator A}
A = 
\begin{pmatrix}
0 & \int_{\mathbb{T}^2}\cdot \partial_\theta u_0\\
S_c^{-1} \partial_\theta u_0 & \Id-cS_c^{-1}+S_c^{-1}\partial_x(u_0\cdot)\\
\end{pmatrix}
,
\end{equation}
\begin{equation*}
\tilde e = 
\begin{pmatrix}
0 \\
-S_c^{-1} e
\end{pmatrix}
,
\end{equation*}
and
\begin{equation*}
\tilde N
\begin{pmatrix}
\sigma\\
\delta
\end{pmatrix}
=
-S_c^{-1}\left(\frac12 \partial_x(\delta^2)+\sigma\partial_\theta
\delta\right)
.
\end{equation*}

\subsection{Algorithm for computing periodic orbits.}\label{section: Newton
scheme}

From the discussion in Section \ref{section: invariance equation}, we have that
our solution $z=(\sigma, \delta)$ satisfies a functional equation of the form
\begin{equation}\label{eq: smooth equation}
A z = \tilde e+\tilde N(z,z), 
\end{equation}
where $A$, given by Equation \eqref{eq: operator A}, is a bounded linear
operator and $\tilde N$ is the nonlinear part ($N(0)=DN(0)=0$). 

The Newton scheme is based on solving Equation \eqref{eq: smooth equation}
numerically. Given an initial guess $(f_0, u_0)$, we update it by finding the
correction $z = (\sigma, \delta)$ that is a solution of the linear equation
\begin{equation}\label{eq: smooth linear system}
A z = \tilde e, 
\end{equation}
and obtain $(f_1, u_1)=(f_0+\sigma, u_0+\delta)$.  This process is repeated
several times until a stopping criterion, $\|\tilde e(f_k, u_k)\| < tol$, is
fullfilled. As in all Newton's methods, if $(f_k, u_k)$ is an approximate
solution, then at each step the error decreases quadratically,  $\|\tilde
e(f_{k+1}, u_{k+1})\|\approx\|\tilde e(f_k, u_k)\|^2$.  Since the problem is
infinite dimensional, truncation to the most significatives Fourier modes is
required.  This transforms the problem to a finite dimensional one.

Summarizing, we obtain the following algorithm:

\begin{algorithm}{\ }
\begin{itemize}
\item[\textbf{Input}] 
\begin{itemize}
\item An approximate solution $(f_0, u_0)$ of the 
invariance equation \eqref{eq: invariance}.
\item The accuracy $\text{\bf tol}$ for the computation 
of the solution.
This gives an upper bound of the accuracy of the outputs of the algorithm. 
\end{itemize}
\item[\textbf{Output}] 
An approximate solution $(f_k, u_k)$ of the invariance equation 
with tolerance less than $\text{\bf tol}$.
\item[0.a)] Fix a norm
on the space of periodic functions on the torus (see Appendix \ref{section:
appendix} for examples of such norms). 
\item[0.b)] 
Set $k=0$.
\item[1)] Compute the error 
$\tilde e_k = -S_c^{-1}\left(f_k\partial_\theta u_k+L u_k 
+\frac12\partial_x(u_k^2)\right)$.
\item[2)] If $\|\tilde e_k\| < \text{\bf tol}$ stop the algorithm. 
The pair $(f_k, u_k)$ is the approximation of the frequency and the periodic
orbit with the desired accuracy.
\item[3)] Solve the (finite dimensional truncated) linear 
system \eqref{eq: smooth linear system} by means 
of a linear solver, obtaining the solution pair $(\sigma, \delta)$.
\item[4)] Set $u_{k+1}=u_k+\delta_k$, $f_{k+1}=f_k+\sigma_k$, and 
update $k$ with $k+1$.
\item[5)] If $\|(\sigma, \delta) \| < \text{\bf tol}$, stop the algorithm. 
The pair $(f_{k+1}, u_{k+1})$ is
the approximation of the periodic orbit and its frequency with 
the desired accuracy.
Otherwise, repeat the process starting from step 1).

\end{itemize}
\end{algorithm}

\begin{rem}
As said before, all computations are performed by representing all Fourier
series as Fourier polynomials of order, say, $N$. However, we notice that in
Step 1), where the error $\tilde e_k$ is computed, the computation of $u^2$ is
required, hence, when we apply the functional to 
a polynomial of degree $N$, we obtain a polynomial of 
degree $2N$.  We have observed in our numerical tests that a way to 
obtain sharp estimates is to compute the functional with $2N$ coefficients. 
By doing so the bounds  obtained by the algorithm are
very sharp and suitable for the validation scheme presented in Section
\ref{section: implementation theorem}.
\end{rem}

\subsection{Computation of the stability of a periodic 
orbit.}\label{subsection: computation stability}

Once a periodic orbit $(f, u)$ is computed one often desires to compute its
stability (the dimension of the unstable manifold). One way of computing it is
counting the number of eigenvalues of the Floquet operator that are outside the
unit circle.  That is, integrate the linear differential equation
\begin{equation*}
f \partial_\theta v = -Lv- \partial_x(u\cdot v)
\end{equation*}
with initial condition $v_0 = \Id$, up to time $1$, and compute the spectrum of
$v_1$. Then, check how many eigenvalues are outside the unit disk. Of course,
this should be done by truncating all computations in finite dimensions and
bounding the errors.

Another way is computing the spectrum of the unbounded (but closed) linear
operator 
\begin{equation}\label{eq: stability operator}
f \partial_\theta +L+\partial_x(u\cdot )
.
\end{equation}
Given an eigenvalue $\lambda$ of the Floquet operator, $\log(\lambda)+i f\cdot
n$, $n\in\mathbb{Z}$, is an eigenvalue of the operator \eqref{eq: stability
operator}. Hence, restricting the spectrum on a set of the form
$\Gamma_a=\{z\in\mathbb{C} : a\leq \text{Im}(z) < a+f\}$ is in one-to-one
correspondence with the spectrum of the Floquet operator. In particular,
computing the dimension of the unstable manifold is the same as computing the
number of eigenvalues of the operator \eqref{eq: stability operator} restricted
to the left half-plane $\Gamma_a\bigcap \{z: \text{Im}(z) < 0\}$.

Even if the two methods are equivalent for the equations that define a 
differentiable flow, we note that the  method based on studying the 
spectrum of \eqref{eq: stability operator} makes sense even in equations 
that do not define a flow. Hence, this is the method that we will use. 

\section{An a posteriori theorem for the rigorous verification 
of the existence of periodic orbits.}\label{section: a posteriori theorems}

In this section we present an a posteriori result, 
Theorem~\ref{thm: contraction 2} that, given an approximate
solution $(f, u)$ of Equation \eqref{eq: invariance} satifying 
some explicit quantitative assumptions, ensures the existence of 
a true solution  of 
\eqref{eq: invariance} and estimates the 
distance between this true solution and the 
approximate one.  Of course, the solutions of 
\eqref{eq: invariance} give periodic solutions of the 
evolution equation.
Theorem~\ref{thm: contraction 2} is a tailored version of Theorem 2.3
apprearing in \cite{ks_theoretical}.  For the sake of completeness, we will
state it here adapted to Equation \eqref{eq: smooth equation}.
Note that the theorem is basically an elementary contraction mapping 
principle, but that we allow for the application of a preconditioner, which 
makes it more applicable in practice. 

\begin{thm}
\label{thm: contraction 2}
Consider the operator  
\begin{equation}
\label{eq: fixed point}
F(z)=A z - \tilde e-N_{\bu}(z), 
\end{equation}
defined in $\overline{B_{\bu}(\rho)} = \{ u : \|u-\bu\| \le \rho
\}$, with $\rho > 0$, and where  
$N_{\bu}$ the nonlinear part of the operator $F$  at the point $\bu$, that is
\[
N_{\bu}(z)=F(\bu+z)-F(\bu)-DF(\bu)z.
\]
Let $B$ be a linear operator such that $BDF(\bu)$, $BF$
and $BN_{\bu}$ are continuous operators.  If, for some $b, K > 0$ we have:

\begin{enumerate}[label=(\alph*),ref=(\alph*)]

\item
\label{item: 1} 
$\|I-BDF(\bu)\| = \alpha < 1$.

\item 
\label{item: 2} 
$\|B\left(F(\bu)+N_{\bu}(z)\right)\|\leq b$ whenever $\|z\| \leq \rho$.

\item 
\label{item: 3} 
$\text{Lip}_{\|z\| \leq \rho} B N_{\bu}(z) < K$

\item 
\label{item: 4} 
$\frac{b}{1-\alpha} < \rho$.

\item 
\label{item: 5} 
$\frac{K}{1-\alpha} < 1$. 
\end{enumerate}

then there exists $\delta u$ such that $\bu+\delta u$
is in $\overline{B_{\bu}(\rho)}$ and  
is a unique solution of Equation 
\eqref{eq: fixed point}, 
with $\|\delta u\| \leq \frac{\|B F(\bu)\|}{1-\alpha-K}$.

\end{thm}

We are now in a position to write down the theorem for the existence and local
uniqueness of periodic orbits and their period  for the Kuramoto-Sivashinsky
equation. This theorem has been written for the special case of the family of
Banach spaces $X_M$, that depends on the parameters $r, s_1, s_2 \geq 0$. 
It is the Banach space of periodic functions 
$u(\theta, x) = \sum_{(k_1, k_2)\in\mathbb{Z}^2} u_{k_1,k_2} 
e^{i (k_1\cdot x+k_2\cdot \theta)}$ with finite norm
\begin{equation*}
\|u\|_{M} = \sum_{(k_1, k_2) \in\mathbb{Z}^2} M(k_1, k_2)|u_{k_1, k_2}|,
\end{equation*}
where 
\[
M(k_1, k_2)=(1+|k_1|)^{s_1}(1+|k_2|)^{s_2} e^{r (|k_1|+|k_2|)}.
\]
When there is no confusion, we will denote the norm by $\|\cdot \|_{M}$.  These
spaces have the property that all their elements are analytic functions for
$r\neq 0$.  See Appendix \ref{section: appendix} for a more detailed
discussion.

\begin{thm}\label{thm: contraction 3}
Let $r, s_1, s_2 \geq 0$ define the Banach space $X_M$, and $(f, u)$ be an
approximate solution of Equation \eqref{eq: invariance}, with error $e$, and
consider Equation \eqref{eq: smooth equation} for the correction $(\sigma,
\delta)$. Let $B=\Id+\hat B$ be a linear operator, and suppose that the
following conditions are satisfied:
\begin{enumerate}
\item[A)] $\|\hat B\hat A+\hat A+\hat B\|_{M} = \alpha < 1$,

\item[B)] $\|B\|_{M}\|\tilde e\|_{M} \leq e_1$,

\item[C)] $\|B\|_{M}
(\|S_c^{-1} \partial_\theta\|_{M}
+
\frac12\|S_c^{-1}\partial_x\|_{M}
)\leq e_2$,

\item[D)] $(1-\alpha)^2-4 e_1 e_2 > 0$,
\end{enumerate}
then there exists a solution $z_*=(\sigma_*, \delta_*)$ of Equation \eqref{eq:
smooth equation} satisfying $\|z_*\|_{M} \leq E=\frac{e_1}{1-\alpha -
\rho_-}$, where $\rho_- = 1-\alpha-\sqrt{(1-\alpha)^2-4e_1e_2}$.
\end{thm}

\begin{proof}
Let $\rho > 0$ such that $1-\alpha-\sqrt{(1-\alpha)^2-4e_1e_2} < 2e_2\rho <
1-\alpha$.  We need to deduce all the conditions in Theorem \ref{thm:
contraction 2}.  Notice that $z = (\sigma, \delta)$ and $Q(z, z) =
S_c^{-1}(\sigma \partial_\theta\delta+\frac12\partial_x(\delta^2))$.  Condition
1) in Theorem \ref{thm: contraction 2} is the same as condition a) for the
present theorem. 

$b$ in condition b) is $e_1+e_2\rho^2$, because if 
$\|(\sigma, \delta)\|_{M}\leq \rho$,
then 
\begin{equation*}
\begin{split}
& \left\|B\left(\tilde e+S_c^{-1}(\sigma \partial_\theta\delta+
\frac12\partial_x(\delta^2))\right)\right\|_{M}
\leq
\|B\|_{M}\|\tilde e\|_{M}
+\|B S_c^{-1}(\sigma \partial_\theta\delta+\frac12\partial_x(\delta^2))\|_{M}
\\
&\phantom{AAAAA}\leq
e_1
+\|B S_c^{-1}(\sigma \partial_\theta\delta)\|_{M}
+\frac12\|B S_c^{-1}(\partial_x(\delta^2))\|_{M}
\\
&\phantom{AAAAA}\leq e_1+
\|B\|_{M}\left(
\|S_c^{-1}\partial_\theta\|_{M}
|\sigma| \|\delta\|_{M}
+
\frac12\|S_c^{-1}\partial_x\|_{M}
\|\delta\|^2_{M}
\right)
\\
&\phantom{AAAAA}\leq 
e_1+e_2 \rho^2.
\end{split}
\end{equation*}

$K$ in condition c) is $2e_2\rho$ because if $\|z\|_{M}, 
\|\hat z\|_{s_1, s_2}\leq \rho$,
then 
\begin{equation*}
\begin{split}
\|B\left(Q(z, z)-Q(\hat z, \hat z)\right)\|_{M}
\\
\leq \|B\|_{M}\|S_c^{-1} 
\left(\sigma \partial_\theta \delta-\hat\sigma\partial_\theta\hat\delta+
\delta\partial_x\delta
-\hat\delta\partial_x\hat\delta\right)\|_{M}
\\
\phantom{AA}\leq \|B\|_{M}
\left(\|S_c^{-1}(\sigma \partial_\theta (\delta-\hat\delta)+
(\sigma-\hat\sigma)\partial_\theta\hat\delta)\|_{M}
+\|S_c^{-1}(\delta\partial_x(\delta-\hat\delta)+
(\delta-\hat\delta)\partial_x\hat\delta)\|_{M}\right)
\\
\phantom{AA}\leq \|B\|_{M}
\left(
\|S_c^{-1}\partial_\theta\|_{M}
+\frac12\|S_c^{-1}\partial_x\|_{M}\right)
2\rho\|z-\hat z\|_{M}
\\
\phantom{AA}\leq 2e_2\rho\|z-\hat z\|_{M}.
\end{split}
\end{equation*}

Conditions d) and e) of Theorem \ref{thm: contraction 2} are equivalent to 
\begin{equation*}
\frac{2e_2\rho}{1-\alpha} < 1
\text{ \quad and \quad} \frac{e_1+e_2\rho^2}{1-\alpha} < \rho,
\end{equation*}
which are satisfied because 
$1-\alpha-\sqrt{(1-\alpha)^2-4e_1e_2} < 2e_2\rho < 1-\alpha$. 

Finally, the upper bound on the norm on the solution $\|z_*\|_{M}$ is obtained
by applying Theorem \ref{thm: contraction 2} with ${\rho =
\frac{1-\alpha-\sqrt{(1-\alpha)^2-4e_1e_2}}{2e_2}}$. 
\end{proof}

\begin{rem} Notice that, using the radii polynomial approach, 
we obtain that $\alpha, e_2$ are functions of the radius $\rho$. Therefore, we
obtain a range of radii for which Theorem~\ref{thm: contraction 2} applies. Of
course the largest radius is a better result for the uniquess part and the
smallest radius is a better result fof the distance to the initial guess
\end{rem}

\section{Implementation of the rigorous computer assisted validation 
of periodic orbits for the Kuramoto-Sivashinsky 
equation.}\label{section: implementation theorem}

We use Theorem \ref{thm: contraction 3} and construct an implementation of the
computer assisted validation of periodic orbits. Our initial data, $(f, u)$,
will consist of a real number $f$, a trigonometric polynomial $u$ of degrees
$(d_1, d_2)$ in the variables $(\theta, x)$, and the operator $B=\Id+\hat B$, 
where $\hat B$ is a $2d_1\cdot(2d_2+1)\times 2d_1\cdot(2d_2+1)$ dimensional 
matrix (this operator can be obtained by nonrigorous computations 
by approximating the inverse of the operator $A$).

First of all, notice that the constants $e_1$ and $e_2$ in Theorem \ref{thm:
contraction 3} depend on the diagonal operators $S_c^{-1}$, $\partial_\theta$
and $\partial_x$, and on the norms of $B$ and $\tilde e$. The computation of
the norms of the (diagonal) operators $S_c^{-1}\partial_\theta$ and
$S_c^{-1}\partial_x$, is done in Appendix \ref{section: appendix}, lemma
\ref{lem: banach2}. 

Secondly, the computation of the norms of the operator $B$ and the error
$\tilde e$ can be done with the help of computer assisted techniques because
they are finite dimensional: $\tilde e$ is a trigonometric polynomial of
dimension $2d_1\cdot (2d_2+1)$ and $B=\Id +\hat B$ implies that $\|B\|_{M}\leq
1 +\|\hat B\|_{M}$ ( Note that $\| \hat B\|_M$ is the  norm of a finite dimensional matrix).

Finally, it remains to show how to compute operator norm of 
\begin{equation}
\label{eq: defect matrices}
\|\hat B\hat A+\hat A+\hat B\|_M.
\end{equation}
Since $u$ is a trigonometric polynomial, the operator $\hat A$ is a band 
operator: $\hat A_{i, j}=0$ for $|i| > d_1$ or $|j| > 2d_2+1$. Hence 
$\hat A$ decomposes as the sum of a finite matrix $\hat A_{F}$ of dimensions 
$2d_1\cdot(2d_2+1)\times 2d_1\cdot(2d_2+1)$ and a linear operator $\hat A_I$.
This operator $\hat A_I$ is:
\begin{equation}
\label{eq: high terms operator}
\hat A_I = \mathbb{P}_{(> d_1, > d_2)}\left(-cS_c^{-1}+S_c^{-1}\partial_x(u\cdot)\right)
=-c\mathbb{P}_{(> d_1, > d_2)}S_c^{-1} \mathbb{P}_{(> d_1, > d_2)}+
\mathbb{P}_{(> d_1, > d_2)}S_c^{-1}\partial_x \mathbb{P}_{(> d_1, > d_2)}
(u\cdot)
,
\end{equation}
where $\mathbb{P}_{(\leq d_1, \leq d_2)}$ is the projection operator on the
$d_1\cdot (d_2+1)$-dimensional vector space spanned by the low frequencies
and $\mathbb{P}_{(> d_1, >d_2)} = \Id-\mathbb{P}_{(\leq d_1, \leq d_2)}$.

Hence, 
\begin{equation}
\label{eq: bound operator}
\|\hat B\hat A+\hat A+\hat B\|_M\leq \max\left\{\|\hat B\hat A_{F}+\hat 
A_{F}+\hat B\|_M,
\|\Id+\hat B\|_M\|\hat A_{I}\|_M
\right\}
.
\end{equation}

\begin{rem}
Notation $\mathbb{P}_{(> d_1, >d_2)}$ could be a little bit 
misleading:
It does not mean that it is the projection operator on the high 
frequencies for both variables, but the complementary of the low 
frequencies projection operator.
\end{rem}

The norm $\|\hat B\hat A_{F}+\hat A_{F}+\hat B\|_M$ appearing in the upper
bound \eqref{eq: bound operator} can be estimated with the help of computer
assisted techniques, while the norm $\|\Id+\hat B\|_M\|\hat A_{I}\|_M$ is
split  into  the computation of $\|\Id+\hat B\|_M$ and $\|\hat A_{I}\|_M$. The
former is done as said before, while the latter (the bound of the operator
\eqref{eq: high terms operator}) is bounded above by:
\begin{equation}\label{eq: norm tail}
cK_1+K_2K_3,
\end{equation}
where 
$K_1=\|\mathbb{P}_{(> d_1, > d_2)}S_c^{-1} 
\mathbb{P}_{(> d_1, > d_2)}\|_M$, 
$K_2 = \|\mathbb{P}_{(> d_1, > d_2)}S_c^{-1}\partial_x 
\mathbb{P}_{(> d_1, > d_2)}\|_M$
and $K_3 = \|u\|_M$.
Fixing $c=\frac1\nu$ and with the help of Lemma \ref{lem: banach1} 
we obtain that 

\begin{tabular}{l}
$K_1 = 
\sqrt2\max\left\{\max_{x > d_1}
\left\{\dfrac1{p(x)}\right\}, \dfrac1{f(d_2+1)}\right\},$
\\
$K_2 = 
\sqrt2 \left(\frac{4}{3\nu}\right)^{\frac14}
\left(\max\left\{\max_{x > d_1}
\left\{\dfrac1{p(x)}\right\}, \dfrac1{f(d_2+1)}\right\}\right)^{\frac34},$
\\
$K_3 = \sup_{(i_1, i_2)\in\mathbb Z^2} |u_{i_1, i_2}|M(i_1, i_2).$
\end{tabular}

\begin{rem}
Since $u$ is a trigonometric polynomial, $K_3$ is in fact computed by
\[
\sup_{(i_1, i_2)\in [0, 2d_1+1]\times [1, 2d_2+1]} |u_{i_1, i_2}|M(i_1, i_2).
\]
\end{rem}

\begin{rem}
The upper bound given in lemma \ref{lem: banach1} tends to zero as 
the number of modes used in the discretization tends to infinity. This assures 
us that this methodology is reliable. 
\end{rem}

\begin{rem}
The computation of the norm \eqref{eq: defect matrices} is very demaning in
terms of computer power effort. Fortunately, not very sharp 
results  are needed. Provided that we can prove that the norm is 
less than $1$, we obtain a contraction. The final result is not 
too afected by the contraction factor. 
On the other hand, the bound on the error $e_1$ in Theorem \ref{thm: contraction
3} has a very direct influence in the error established.

Hence, a good strategy is  to perform the matrix computations with the
lowest dimensions possible and perform 
the estimate of $e_1$ with the highest possible number  
of modes. This relies on the fact that given two functions $u_0$ and $u_1$ with
$\|u_0-u_1\|_M\leq \delta$, then their associated $\hat A_{u_i}$ satisfy that
$\|\hat A_{u_0}-\hat A_{u_1}\|_M\leq \|S_c^{-1}\partial_x\|_M \delta\leq K_2
\delta$. This strategy is reflected in Algorithm
\ref{algor: computation periodic orbit}.

One should also realize that the calculation of the operator $B$ 
does not need to be justified. Some further heuristic approximations 
that reduce the computational effort could be taken (e.g. a Krylov 
method that gives a finite rank approximation). We have not taken advantage of this possibilities since 
they were not needed in our case. 

Finally, we note that since the preconditioner is not so crucial, 
and it is more expensive to compute, in continuation algorithms, 
it could be good to update it less frequently than the residual. 
\end{rem}

Now in a position of giving the algorithm
for the validation of the existence and local uniqueness of periodic orbits
near a given approximate one $(f, u)$. We suppose that the approximation is
obtained by the methods explained in Section \ref{section: Newton scheme}.

\begin{rem}
For more details on the computer implementation of this algorithm 
(e. g. the rigorous manipulation of Fourier series), we refer to the appendix in 
\cite{FiguerasHaro_CAP} or \cite{Haro_Survey}.
\end{rem}

\begin{algorithm}{\ }\label{algor: computation periodic orbit}
\begin{itemize}
\item[\textbf{Input}]
\begin{itemize}
\item $r, s_1, s_2 \geq 0$, defining  the Banach space $X_{M}$.
\item An approximate solution 
$(f, u)$ to Equation \eqref{eq: invariance}, of dimensions
$d_1, d_2$ in the variables $t, x$.
\item A pair of natural numbers $\tilde d_1, \tilde d_2$ such that 
$\tilde d_i \leq d_i$, $i=1,2$.
\end{itemize}
\item[\textbf{Output}] If succeeded, the existence of a constant 
$\rho_- > 0$ where a (unique) solution of the invariance equation 
exists inside the ball centered at $(f, u)$ with radius $\rho_-$.

\item[1)] Compute the trigonometric polynomial $\tilde u$ by 
truncating $u$ up to $\tilde d_1, \tilde d_2$.

\item[2)] Compute an upperbound $\delta$ of $\|\tilde u-u\|_M$.

\item[3)] Compute the matrix $\hat A_{F}$ and the matrix $\hat B$ 
associated to $\tilde u$. 

\item[4)] Compute an upper bound, $\alpha_1$, of 
$\|\hat B\hat A_{F}+\hat B + \hat A_F \|_{M}.$

\item[5)] Compute upper bounds of the constants $K_1, K_2$ and $K_3$.

\item[6)] Compute an upper bound, $\alpha_2$, of $cK_1+K_2K_3$. 

\item[7)] Compute an upper bound, $b$, of $1+\|\hat B\|_{M}$.

\item[8)] Compute $\alpha = \max\{\alpha_1, \alpha_2\}+K_2 \delta b$. If
$\alpha$ is greater than $1$ then the algorithm stops and the result is that
the validation has failed, otherwise continue with Step 7).

\item[9)] Compute an upper bound, $e_0$, of $\|\tilde e\|_{M}$.

\item[10)] Compute an upper bound, $e_1$, of $b\cdot e_0$.
 
\item[11)] Compute an upper bound, $e_2$, of 
$b\cdot(\|S_c^{-1}\partial_\theta\|_M+\frac12 \|S_c^{-1}\partial_x\|_M)$.

\item[12)] Check if $(1-\alpha)^2-4e_1e_2 > 0$.
\end{itemize}
If the inequality in 12)  is true 
then, by Theorem~\ref{thm: contraction 3}, 
there exists a unique periodic orbit $(f_*, u_*)$ 
such that 
$\|(f_*-f, u_*-u)\|_{M}\leq E=\dfrac{e_1}{1-\alpha-\rho_-}$, where 
$\rho_-$ has the expression as in Theorem \ref{thm: contraction 3}.

\end{algorithm}

\begin{rem}
The computation of the product of the interval matrices $\hat B$ and $\hat A_F$
is the bottleneck, in terms of computational time, of the algorithm: naive
multiplication of the matrices leads to disastrous speed results. To speed up
this we use the techniques in \cite{Rump_matrices_0, Rump_matrices}, which
describe algorithms for the rigorous computation of product of interval
matrices with the help of the \verb$BLAS$ package. See Appendix \ref{section:
multiplication} for a presentation of this technique.
\end{rem}

\subsection{Improving the radius of analyticity of solutions}
\label{subsection: improving}

A simple strategy for giving rigorous lower bounds of the analyticity radius of
the solutions is by first performing Algorithm \ref{algor: computation periodic
orbit} with $r\approx 0$. Then, apply a posteriori bounds for improving the
value of $r$. (See  \cite{Hungria_Lessard_Mireles-James_2016} for an
application of this technique in the context of ODEs).

Denote by $\alpha_r$, $e_{1, r},$ and $e_{2, r}$ the upperbounds appearing in
Theorem \ref{thm: contraction 3} when computed with the one-parametric norm
$\|\cdot \|_{M_r}$ (the weight $M_r$ depends on the radius of analyticity).
Moreover, notice that for any trigonometric polynomial $u$ of dimensions
$d_1\times d_2$ and for any $\hat r > 0$ we have $\|u\|_{M_{\hat r}}\leq
\|u\|_{M_0}e^{\hat r d_1 d_2}$ and for any finite dimensional $A$ of dimensions
$d\times d$ $\|T\|_{M_{\hat r}}\leq \|T\|_{M_0}e^{\hat r d}$.  Hence, since the
application of Theorem \ref{thm: contraction 3} is performed with finite
dimensional approximations we obtain that $\alpha_{\hat r}\leq \alpha_0 e^{\hat
r d_1 d_2}$, $e_{1,\hat r}\leq e_{1,0}e^{2\hat r d_1 d_2}$ and $e_{2, \hat
r}\leq e_{2, 0} e^{\hat r d_1 d_2}$. So, by imposing that these upperbounds
satisfy the conditions appearing in Theorem \ref{thm: contraction 3} we obtain
larger values of the radius of analyticity of the solutions.

Of course, some more detailed results could be obtained by repeating the 
calculation of the norms in the spaces of analytic spaces closer to 
the true value. Of course, this will require reduing all the estimates of 
norms. 

The analyticity properties of solutions of K-S equations have been studied
rigorously in \cite{Colletanalyticity, Grujic00} and it is shown to have
thermodynamicas properties and relations with the number of zeros, determining
modes, etc. 

\section{Some numerical examples.}
\label{section: numerical examples}

In this section we present some examples of the methods developed in Section
\ref{section: Newton scheme} for the computation of periodic orbits and in
Section \ref{section: implementation theorem} for the a posteriori verification
of them.

\subsection{Example of numerical computation. Period doubling.}

We have continued some branches of the doubling period bifurcation diagram,
shown in Subfigure a) in Figure \ref{figure: sections po cascade}. This has
been done by first computing some of the attracting orbits by integration, see
Section \ref{section: introduction}. These periodic orbits have been used as
seeds for our numerical algorithm. 

Specifically, for the values of the parameter $\frac 1\nu$ equal to $32.9$,
$33.1$ and $33.3$ we have computed 3 (attracting) periodic orbits at the first
3 stages of the period doubling cascades. Then, for each one of them, we have
continued them with our numerical algorithm.  With the help of Algorithm 1
we have been able to cross the period doubling bifurcations, where the
attracting orbits bifurcate to a doubled period one (that is attracting) and to
an unstable one. Our continuations are able to continue these unstable orbits.
See Figure \ref{figure: cascade with periodic orbits} for a representation of
these orbits in the period doubling cascade diagram and Figure \ref{figure:
sample periodic orbits} for the representation of two of these orbits.

The computational time of its validation takes no more than 30 seconds in a
single 2.7 GHz CPU of a regular laptop. We hope that thid could be used 
in the catalogue of periodic orbits computed in \cite{LanCvitanovic2008}. 
Note that, of course, validating different peridic orbits is 
verily easily paralellizable. 

\begin{figure}
\centering
\resizebox{100mm}{!}{\includegraphics[angle=270]
{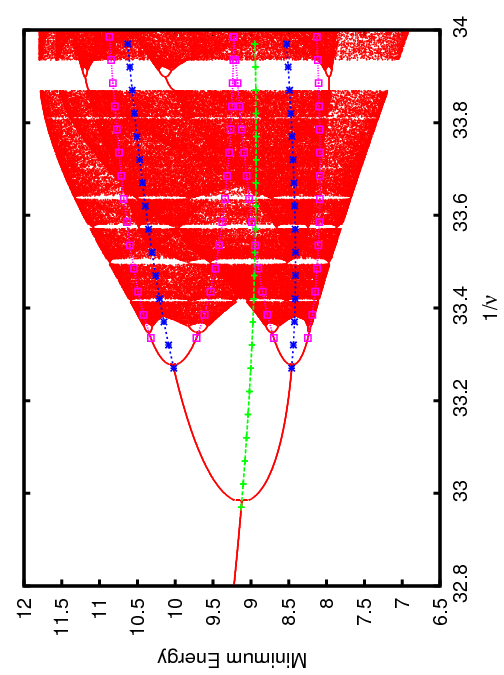}}
\caption{The continuation of the first 3 periodic orbits on the 
period doubling cascade.
These are superposed to the period doubling cascade using three 
different colors.}
\label{figure: cascade with periodic orbits}

\end{figure}

\begin{figure}
\resizebox{170mm}{!}{
\begin{tabular}{cc}
\resizebox{85mm}{!}{\includegraphics[type=png,ext=.png,read=.png,angle=270]
{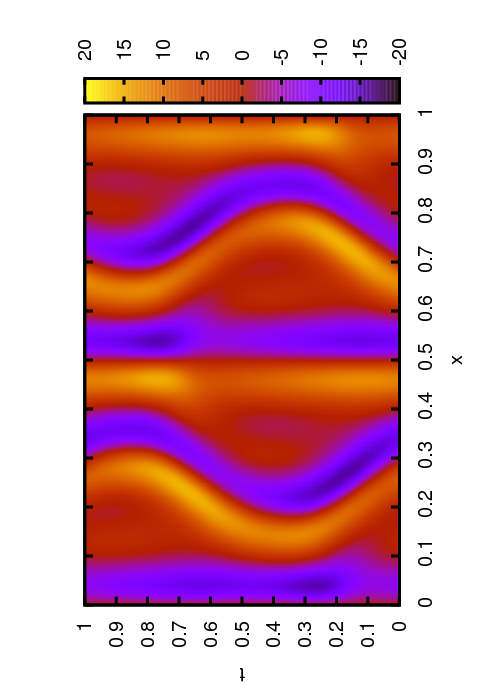}}&
\resizebox{85mm}{!}{\includegraphics[type=png,ext=.png,read=.png,angle=270]
{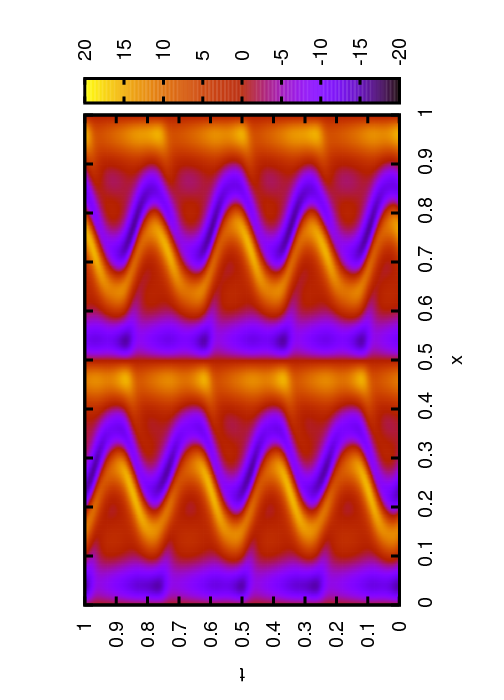}}\\
{$1/\nu = 33.27$. Period $=0.89893314191428$.}&
{$1/\nu = 33.5069$. Period $=3.49489164733297$.}\\
\end{tabular}
} \caption{Representation of two periodic orbits computed with the Newton
method. Colors represent the value of the orbit $u$ at the $(\theta, x)$
coordinate.}
\label{figure: sample periodic orbits}
\end{figure}

\subsection{Example of a validation.}

With the help of the algorithm presented in Section \ref{section:
implementation theorem} (and with the improvment trick explained in Subsection
\ref{subsection: improving}) we can validate the existence of some periodic
orbits.  For example, we validate the existence of a periodic orbit with $\frac
1\nu = 32.97$. The approximate data is given by $40$ $x$ modes and $(2\cdot
19+1)$ $t$ modes. The approximate period is $0.895839$. The validation has
been done with the finite matrix $\hat B$ with dimension 6994, and with $r =
s_1 = s_2 = 1\e{-12}$. 

The  output of the validation is:

\begin{itemize}
\item 
The error produced by the approximate periodic orbit is 
$\|S^{-1}_c \varepsilon\|_M \leq 8.489632\e{-10}$.

\item 
$K_1 \leq 8.189680\e{-3}$
,
$K_2 \leq 6.332728\e{-2}$
,
$K_3 \leq 6.693947$
.
\item
The error of the tails of the operator $\hat A$ is less than or equal to
$K_1\cdot c+K_2\cdot K_3 = 6.946684\e{-1}$.

\item
The norm of the approximate inverse of the linear operator $Id+\hat A$ is 
$\|\text{Id}+\hat B\|_M = 4.567111\e{1}$.

\item 
$\|\hat B+\hat A+\hat B \cdot \hat A\|_M = 6.716849\e{-14}$

\item
$\alpha = \|\text{Id}+\hat B + \hat A + \hat B\hat A\|_M = 6.946684\e{-1}$

\item 
$(1-\alpha)^2-4 e_1 e_2 = 9.321303\e{-2}$

\end{itemize}

As a result of the validation we obtain that the distance of the true periodic
orbit to the approximate solution is less than or equal to $1.269966\e{-7}$.

The computational time of one of these validations is no more than 1017 seconds
in a single 2.7 GHz CPU on a regular laptop.

Other validation results for other periodic orbits is shown in table
\ref{table: validation_results}.

\begin{table}
\begin{center}
\tiny
{
\begin{tabular}{| l | l | l | l | l | l | l | l |}
\hline
$\frac 1\nu$ &Period & 
$E$& 
Improved radius of analyticity&
Improved $E$
\\
\hline
$8.199953$ &
$2.992730$ &
$2.463363\e{-11}$ &
$1.512026\e{-4}$  & 
$1.002122\e{-9}$
\\
\hline
$8.230453$ &
$3.074450$ &
$4.385135\e{-11}$ &
$1.268306\e{-4}$ &
$6.199587\e{-9}$ 
\\
\hline
$31.00000$ &
$0.806901$ &
$8.642523\e{-10}$ &
$9.154580\e{-5}$ &
$1.043836\e{-8}$
\\
\hline
$32.97000$ &
$0.895839$ &
$1.269966\e{-7}$ &
$1.101236\e{-4}$ &
$2.902773\e{-6}$
\\
\hline
$33.27010$ &
$0.881170$ &
$1.049117\e{-7}$ &
$1.120727\e{-4}$ &
$4.092897\e{-6}$
\\
\hline
\end{tabular}
\caption{Validation results of some periodic orbits for different values of the
parameter $\nu$. The columns show: $\frac{1}\nu$, the period of the periodic
orbit, the radius of the ball obtained from Theorem \ref{thm: contraction 3}
computed with the norm $\|\cdot\|_{M}$ with $s_1=s_2=10^{-12}$ and
$r=10^{-12}$, the improved radius of analyticity obtained applying the trick
explained in Remark \ref{subsection: improving}, and the new radius of the ball
obtained from Theorem \ref{thm: contraction 3} computed with the norm
$\|\cdot\|_{M}$ with $s_1=s_2=10^{-12}$ and $r$ equal to the value in fourth
column. All validation took around 1000 seconds in a single 2.7 GHz CPU on a
regular laptop. Some of the periodic orbits that appear in this table appear
also in \cite{Piotr3}.} \label{table: validation_results} }
\end{center}
\end{table}

\appendix

\section{$L^1$ weighted spaces of periodic functions.}\label{section: appendix}

In this section we develop the theoretical framework developed in
\cite{ks_theoretical} for some concrete spaces.  The spaces have been chosen
have the properties that they  can encode analytic functions, the norms are
easily computable from Fourier series, have Banach algebra properties and the
norms of linear operators can be estimated easily from the matrix elements. A
technical, but sometimes useful property is that the dual is also a sequence
space so that all the functionals are represented by its matrices (that is,
there are no \emph{functionals at infinity}. See \cite{FontichLM11} for
examples and results on spaces based on $L^\infty$ which have functionals at
infinity such as the taking the limit). 

From now on, all norms of vectors in finite dimensional vector spaces will 
be the $L^1$-norm, 
in particular, the norm of a complex number is $|a+bi|=|a|+|b|$.

Let $X_{M}$, $r, s_1, s_2 \geq 0$, be the Banach space of periodic functions 
$u(\theta, x) = \sum_{(k_1, k_2)\in\mathbb{Z}^2} u_{k_1,k_2} 
e^{i (k_1\cdot x+k_2\cdot \theta)}$ with finite norm
\begin{equation*}
\|u\|_{M} = \sum_{(k_1, k_2) \in\mathbb{Z}^2} M(k_1, k_2)|u_{k_1, k_2}|,
\end{equation*}
where 
\[
M(k_1, k_2)=(1+|k_1|)^{s_1}(1+|k_2|)^{s_2} e^{r (|k_1|+|k_2|)}.
\]
Usually the values of $r, s_1, s_2$ are fixed. When there is no confusion, we
will denote the norm by $\|\cdot \|_{M}$, otherwise 
we will remark the dependency on the parameters $r, s_1, s_2$ by 
subscripts on the weight $M$ (e.g. $M_r$).

If $r> 0$ then $X_M$ is a subspace
of the analytic functions with complex band radius $r$. Also, if $s_1\leq
s_1'$, $s_2\leq s_2'$ and $r\leq r'$, then $X_M\subset X_{M'}$ and $\|u\|_M\leq
\|u\|_{M'}$.

\begin{rem}
For notational purposes, we present all the theory and analytic computations in
$X_{M}$ with the complex exponential basis, even though all the periodic
functions we work with are real. For accuracy efficiency, 
we implement our codes with the sine-cosine basis.  That is, the
periodic functions are represented as
\begin{equation*}
u(\theta, x) = \sum_{k_1 =1}^{\infty}\left(\sum_{k_2=0}^{\infty}
a_{k_1, k_2}\cos(k_2\theta)+b_{k_1, k_2}\sin(k_2\theta)\right)sin(k_1x).
\end{equation*}
With this basis, the norm defined above is  
\begin{equation*}
\|u\|_{M} = \sum_{1\leq k_1 < \infty, 0\leq k_2 < \infty}
\left(|a_{k_1, k_2}|+|b_{k_1, k_2}|\right)M(k_1, k_2).
\end{equation*}
For this choice of norm, all the estimates computed with the exponential 
basis remain valid 
with the sine-cosine basis.
\end{rem}

$X_{M}$ is a Banach algebra, 
$\|u\cdot v\|_{M} \leq \|u\|_{M} \|v\|_{M} $. This is a consequence 
of the fact that the weight $M$ is submultiplicative, 
$M(k_1+l_1, k_2+l_2) 
\leq M(k_1, k_2)\cdot M(l_1, l_2)$. 

If $T\colon X_{M}\rightarrow X_{M}$ is a linear operator with coefficients 
$\left\{T_{i, j}\right\}_{(i_1, i_2), (j_1, j_2)\in\mathbb{Z}^2}$, then 
its norm is 
\begin{equation*}
\|T\|_{M} = 
\sup_{(j_1, j_2)\in\mathbb{Z}^2} \dfrac{\sum_{(i_1, i_2)\in\mathbb{Z}^2}
|T_{(i_1, i_2), (j_1, j_2)}|M(i_1, i_2)}{M(j_1, j_2)}
.
\end{equation*}

Hence, for example, the norm of the multiplication operator 
$v\longrightarrow u\cdot v$ is 
\begin{equation}\label{eq: norm multiplication operator}
\sup_{(i_1, i_2)\in\mathbb Z^2} |u_{i_1, i_2}|M(i_1, i_2).
\end{equation}

\begin{rem}
The norm in \eqref{eq: norm multiplication operator} is sharper than the norm
estimates that used the property that $X_M$ is a Banach algebra.
\end{rem}

\begin{rem}
These norms scale very well with respect the weights: Given a trigonometric 
polynomial $u$ of dimensions $d_1\times d_2$, and $\hat r > r$, then 
$\|u\|_{M_{\hat r}}\leq \|u\|_{M_{r}}e^{(\hat r-r)d_1 d_2}$. Similarly, 
for a finite dimensional matrix $T$ of sizes $d\times d$, 
$\|T\|_{M_{\hat r}}\leq \|T\|_{M_r}e^{(\hat r-s)d}$.
\end{rem}

\subsection{Two lemmas for the validation algorithm.}

The following two lemmas are using for the computation of some preliminary 
estimates for the validation 
algorithm in Section \ref{section: implementation theorem}.
\begin{lem}\label{lem: banach1}
Let $p(x)=\nu x^4-x^2+c$ and  
$S_c = f\partial_\theta+\nu\partial_x^4+\partial_x^2+c\Id$, 
with $c = \frac1\nu$. 
Then it is satisfied that
\begin{enumerate}
\item 
\begin{equation*}
\|\mathbb{P}_{(>d_1, >d_2)}S_c^{-1}\mathbb{P}_{(>d_1, >d_2)}\|
_{M}\leq
\sqrt2\max\left\{\max_{x > d_1}
\left\{\dfrac1{p(x)}\right\}, \dfrac1{f(d_2+1)}\right\}.
\end{equation*}
\item 
\begin{equation*}
\|\mathbb{P}_{(>d_1, >d_2)}S_c^{-1}\partial_x\mathbb{P}_{(>d_1, >d_2)}\|
_{M}\leq
\sqrt2 \left(\frac{4}{3\nu}\right)^{\frac14}
\left(\max\left\{\max_{x > d_1}
\left\{\dfrac1{p(x)}\right\}, \dfrac1{f(d_2+1)}\right\}\right)^{\frac34}.
\end{equation*}

\end{enumerate}
\end{lem}
\begin{proof}
The first one follows from the fact that an upper bound of 
\begin{equation*}
\max_{x > d_1 \text{ or } y > d_2}\dfrac{fy+p(x)}{(fy)^2+p(x)^2},
\end{equation*}
is a consequence of the inequalities
\begin{equation*}
\begin{array}{ll}
\dfrac{fy+p(x)}{(fy)^2+p(x)^2}\leq \dfrac{\sqrt{2}}{fy}
&\\
\dfrac{fy+p(x)}{(fy)^2+p(x)^2}\leq \dfrac{\sqrt{2}}{p(x)}
&
\end{array}
.
\end{equation*}

The second upper bound follows in a similar way as the first one 
but considering that  
\begin{equation*}
\begin{array}{l}
\max_{x > d_1 \text{ or } y > d_2}\dfrac{fy+p(x)}{(fy)^2+p(x)^2}x\\ 
\leq \max_{x > d_1 \text{ or } y > d_2}
\dfrac{fy+p(x)}{(fy)^2+p(x)^2}\frac{x}{p(x)^{\frac14}}p(x)^{\frac14}\\ 
\leq \left(\dfrac{4}{3\nu}\right)^{\frac14}
\max_{x > d_1 \text{ or } y > d_2}
\dfrac{fy+p(x)}{(fy)^2+p(x)^2}p(x)^{\frac14}.\\ 
\end{array}
\end{equation*}
\end{proof}

\begin{lem}\label{lem: banach2}
Let $p(x)=\nu x^4-x^2+c$ and  
$S_c = f\partial_\theta+\nu\partial_x^4+\partial_x^2+c\Id$, 
with $c = \frac1\nu$. 
Then it is satisfied that
\begin{enumerate}
\item 
\begin{equation*}
\|S_c^{-1}\partial_x\|
_{M}\leq
\sqrt2 \left(\frac{4}{3\nu}\right)^{\frac14}.
\end{equation*}
\item 

\begin{equation*}
\|S_c^{-1}\partial_\theta\|
_{M}\leq
\frac{\sqrt2}{f}.
\end{equation*}

\end{enumerate}
\end{lem}

\begin{proof}
The first upper bound follows from the observation that 
\begin{equation*}
\dfrac{fy+p(x)}{(fy)^2+p(x)^2}x
\end{equation*}
is less than or equal to 
\begin{equation*}
\left(\frac{4}{3\nu}\right)^{\frac14}
\dfrac{fy+p(x)}{(fy)^2+p(x)^2}p(x)^{\frac14}.
\end{equation*}
The second one is proved similarly.

\end{proof}

\section{Fast interval matrix multiplication algorithms.}
\label{section: multiplication}

The naive multiplication of two high dimensional ($\approx 5000\times5000$)
interval matrices is very inefficient.  Here we present a fast algorithm that
we have used for (rigorously) multiplying interval matrices.  It is Algorithm
4.5 in \cite{Rump_matrices_0}.

The algorithm relies on a clever usage of the fast \verb$double$ floating point
based matrix multiplication software \verb$BLAS$ \cite{blas3}.  It combines the
fast algorithms in \verb$BLAS$ and rounding flags.

Let $\boldsymbol{A} = [A_1, A_2]$ and $\boldsymbol B=[B_1, B_2]$ be two
interval matrices, where inside brackets we have written the lower and upper
point matrices.  (In bold letters we denote interval matrices, and in plain
ones \verb$double$ matrices).  The algorithm produces an interval matrix
$\boldsymbol C=[C_1, C_2]$ satisfying $\boldsymbol A\cdot \boldsymbol
B\subseteq\boldsymbol C$. 

\begin{algorithm}{\ }
\begin{itemize}
\item[1)] Set the rounding up.
\item[2)] Compute the \verb$double$ matrices:
\begin{itemize}
\item[] $mA=(A_1+A_2)/2$.
\item[] $rA=mA-A_1$.
\item[] $mB=(B_1+B_2)/2$.
\item[] $rB=mB-B_1$.
\item[] $rC=|mA|\cdot rB+rA\cdot(|mB+rB|)$, where $|\cdot|$ denotes the 
matrix with absolute values 
in all its entries.
\end{itemize}
\item[3)] $C_2=mA\cdot mB+rC$.
\item[4)] Set the rounding down.
\item[5)] $C_1=mA\cdot mB-rC$.
\end{itemize}
\end{algorithm}

All multiplications should be done using \verb$BLAS$.  Notice that this
algorithm requires 4 \verb$double$ matrix multiplications.

\section*{Acknowledgments}
R. L. was partially supported by NSF grant DMS-1500493. J.-L. F.  was partially
supported by Essen, L. and C.-G., for mathematical studies. We are very
grateful to W. Tucker for pointing out the existence of the fast interval
matrix multiplication algorithms. We are also grateful for useful discussions
with G. Arioli, P. Cvitanovic, M. Gameiro, J. Gomez-Serrano, A.Haro, H. Koch,
J.-P. Lessard, K.  Mischaikow, W. Tucker and P. Zgliczynski.

\bibliography{./bibliography}{} 

\begin{thebibliography}{10}

\bibitem{ArioliKoch1}
G.~Arioli and H.~Koch.
\newblock {Computer-assisted methods for the study of stationary solutions in
  dissipative systems, applied to the Kuramoto-Sivashinski equation}.
\newblock {\em Arch. Rational Mech. An.}, 197:1033, 2010.

\bibitem{ArioliKoch2}
G.~Arioli and H.~Koch.
\newblock {Integration of dissipative PDEs: a case study}.
\newblock {\em SIAM J. of Appl. Dyn. Syst.}, 9:1119--1133, 2010.

\bibitem{Armbruster}
D.~Armbruster, J.~Guckenheimer, and P.~Holmes.
\newblock {Kuramoto-Sivashinsky dynamics on the center-unstable manifold}.
\newblock {\em Siam J. Appl. Math.}, 49:676--691, 1989.

\bibitem{Brent73}
R.~P. Brent.
\newblock {\em Algorithms for minimization without derivatives}.
\newblock Prentice-Hall, Inc., Englewood Cliffs, N.J., 1973.
\newblock Prentice-Hall Series in Automatic Computation.

\bibitem{CGL_Boussinesq}
R.~Castelli, M.~Gameiro, and J.-P. Lessard.
\newblock Rigorous numerics for ill-posed pdes: periodic orbits in the
  boussinesq equation.
\newblock Submitted, 2015.

\bibitem{Chueshov02}
I.~D. {Chueshov}.
\newblock {\em {Introduction to the theory of infinite-dimensional dissipative
  systems. Transl. from the Russian by Constantin I. Chueshov, edited by Maryna
  B. Khorolska.}}
\newblock Kharkiv: ACTA, 2002.
\newblock http://www.emis.de/monographs/Chueshov/.

\bibitem{Cohenetaltri1976}
B.~Cohen, J.~Krommes, W.~Tang, and M.~Rosenbluth.
\newblock Nonlinear saturation of the dissipative trapped-ion mode by mode
  coupling.
\newblock {\em Nucl. Fus.}, 16:971--992, 1976.

\bibitem{Colletattracting}
P.~Collet, J.-P. Eckmann, H.~Epstein, and J.~Stubbe.
\newblock {A global attracting set for the Kuramoto-Sivashinsky equation}.
\newblock {\em Commun. Math. Phys.}, 152:203--214, 1993.

\bibitem{Colletanalyticity}
P.~Collet, J.-P. Eckmann, H.~Epstein, and J.~Stubbe.
\newblock {Analyticity for the Kuramoto-Sivashinsky equation}.
\newblock {\em Physica D}, 67:321--326, 1993.

\bibitem{ColletEK81}
P.~Collet, J.-P. Eckmann, and H.~Koch.
\newblock Period doubling bifurcations for families of maps on {${\bf R}^{n}$}.
\newblock {\em J. Statist. Phys.}, 25(1):1--14, 1981.

\bibitem{CFNT_book}
P.~Constantin, C.~Foias, B.~Nicolaenko, and R.~Temam.
\newblock {\em Integral manifolds and inertial manifolds for dissipative
  partial differential equations}, volume~70 of {\em Applied Mathematical
  Sciences}.
\newblock Springer-Verlag, New York, 1989.

\bibitem{CvitanovicDavidchackSiminos2010}
P.~Cvitanovi{\'c}, R.~L. Davidchack, and E.~Siminos.
\newblock On the state space geometry of the {K}uramoto-{S}ivashinsky flow in a
  periodic domain.
\newblock {\em SIAM J. Appl. Dyn. Syst.}, 9(1):1--33, 2010.

\bibitem{LlaveS16}
R.~de~la Llave and Y.~Sire.
\newblock An a posteriori kam theorem for whiskered tori in hamiltonian partial
  differential equations with applications to some ill-posed equations.
\newblock 2016.
\newblock arXiv:1602.03775.

\bibitem{DCCST2016}
X.~Ding, H.~Chaté, P.~Cvitanović, E.~Siminos, and K.~A. Takeuchi.
\newblock Estimating dimension of inertial manifold from unstable periodic
  orbits.
\newblock \url{http://arxiv.org/abs/1604.01859}, 2016.

\bibitem{blas3}
J.~J. Dongarra, J.~Du~Croz, I.~Duff, and S.~Hammarling.
\newblock A set of level 3 basic linear algebra subprograms.
\newblock {\em ACM Trans. Math. Softw.}, 16:1--17, 1990.

\bibitem{EdenFNT94}
A.~Eden, C.~Foias, B.~Nicolaenko, and R.~Temam.
\newblock {\em Exponential attractors for dissipative evolution equations},
  volume~37 of {\em RAM: Research in Applied Mathematics}.
\newblock Masson, Paris; John Wiley \& Sons, Ltd., Chichester, 1994.

\bibitem{Feigenbaum78}
M.~J. Feigenbaum.
\newblock Quantitative universality for a class of nonlinear transformations.
\newblock {\em J. Statist. Phys.}, 19(1):25--52, 1978.

\bibitem{FiguerasHaro_CAP}
J.-L. Figueras and A.~Haro.
\newblock {Reliable Computation of Robust Response Tori on the Verge of
  Breakdown}.
\newblock {\em SIAM J. Appl. Dyn. Syst.}, 11(2):597--628, 2012.

\bibitem{ks_theoretical}
J.-L. Figueras, J.-P. Lessard, M.~Gameiro, and R.~de~la Llave.
\newblock {A framework for the numerical computation and a-posteriori
  verification of invariant objects of evolution equations}.
\newblock 2016.

\bibitem{Foias_Nicolaenko_Sell_Teman_88}
C.~Foias, B.~Nicolaenko, G.~R. Sell, and R.~Temam.
\newblock Inertial manifolds for the {K}uramoto-{S}ivashinsky equation and an
  estimate of their lowest dimension.
\newblock {\em J. Math. Pures Appl. (9)}, 67(3):197--226, 1988.

\bibitem{FontichLM11}
E.~Fontich, R.~de~la Llave, and P.~Mart{\'{\i}}n.
\newblock Dynamical systems on lattices with decaying interaction {I}: a
  functional analysis framework.
\newblock {\em J. Differential Equations}, 250(6):2838--2886, 2011.

\bibitem{GarciaNT98}
B.~Garc{\'{\i}}a-Archilla, J.~Novo, and E.~S. Titi.
\newblock Postprocessing the {G}alerkin method: a novel approach to approximate
  inertial manifolds.
\newblock {\em SIAM J. Numer. Anal.}, 35(3):941--972, 1998.

\bibitem{Grujic00}
Z.~Gruji{\'c}.
\newblock Spatial analyticity on the global attractor for the
  {K}uramoto-{S}ivashinsky equation.
\newblock {\em J. Dynam. Differential Equations}, 12(1):217--228, 2000.

\bibitem{Haro_Survey}
A.~Haro, M.~Canadell, J.-L. Figueras, A.~Luque, and J.-M. Mondelo.
\newblock {\em The Parameterization Method for Invariant Manifolds: from
  Rigorous Results to Effective Computations}.
\newblock Applied Mathematical Sciences. Springer International Publishing,
  first edition, 2016.

\bibitem{Hungria_Lessard_Mireles-James_2016}
A.~Hungria, J.-P. Lessard, and J.~D. Mireles~James.
\newblock Rigorous numerics for analytic solutions of differential equations:
  the radii polynomial approach.
\newblock {\em Math. Comp.}, 85(299):1427--1459, 2016.

\bibitem{Ilyashenko}
Y.~S. Ilyashenko.
\newblock {Global analysis of the phase portrait for the Kuramoto- Sivashinski
  equatio}.
\newblock {\em Journal of dynamics and Dif. Equations}, 4(4):585--615, 1992.

\bibitem{Jolly_Kevrekidis_Titi_90}
M.~S. Jolly, I.~G. Kevrekidis, and E.~S. Titi.
\newblock Approximate inertial manifolds for the {K}uramoto-{S}ivashinsky
  equation: analysis and computations.
\newblock {\em Phys. D}, 44(1-2):38--60, 1990.

\bibitem{JollyRT00}
M.~S. Jolly, R.~Rosa, and R.~Temam.
\newblock Accurate computations on inertial manifolds.
\newblock {\em SIAM J. Sci. Comput.}, 22(6):2216--2238 (electronic), 2000.

\bibitem{Kuramoto76}
Y.~Kuramoto and T.~Tsuzuki.
\newblock {Persistent propagation of concentration waves in dissipative media
  far from thermal equilibrium}.
\newblock {\em Progr. Theoret. Phys.}, 55:356--369, 1976.

\bibitem{LanChandreCvitanovic2006}
Y.~Lan, C.~Chandre, and P.~Cvitanovi{\'c}.
\newblock Newton's descent method for the determination of invariant tori.
\newblock {\em Phys. Rev. E (3)}, 74(4):046206, 10, 2006.

\bibitem{LanCvitanovic2008}
Y.~Lan and P.~Cvitanovi{\'c}.
\newblock Unstable recurrent patterns in {K}uramoto-{S}ivashinsky dynamics.
\newblock {\em Phys. Rev. E (3)}, 78(2):026208, 12, 2008.

\bibitem{laqueyetaltri1975}
R.~LaQuey, S.~Mahajan, P.~Rutherford, and W.~Tang.
\newblock Nonlinear saturation of the trapped-ion mode.
\newblock {\em Phys. Rev. Lett.}, 34:391--394, 1975.

\bibitem{ks_jp_marcio}
J.-P. Lessard and M.~Gameiro.
\newblock {A posteriori verification of invariant objects of evolution
  equations: periodic orbits in the Kuramoto-Sivashinsky PDE}.
\newblock 2016.

\bibitem{Piotr1}
K.~Mischaikow and P.~Zgliczynski.
\newblock {Rigorous Numerics for Partial Differential Equations: the
  Kuramoto-Sivashinsky equation}.
\newblock {\em {Foundations of Computational Mathematics}}, 1:255--288, 2001.

\bibitem{Neuberger10}
J.~W. Neuberger.
\newblock {\em Sobolev gradients and differential equations}, volume 1670 of
  {\em Lecture Notes in Mathematics}.
\newblock Springer-Verlag, Berlin, second edition, 2010.

\bibitem{Nicolaenkoetaltri}
B.~Nicolaenko, B.~Scheuer, and R.~Teman.
\newblock {Some Global Dynamical Properties of the Kuramoto-Sivashinsky
  Equations: Nonlinear Stability and Attractors}.
\newblock {\em Physica D}, 16:155--183, 1985.

\bibitem{NovoTW01}
J.~Novo, E.~S. Titi, and S.~Wynne.
\newblock Efficient methods using high accuracy approximate inertial manifolds.
\newblock {\em Numer. Math.}, 87(3):523--554, 2001.

\bibitem{Rump_matrices}
K.~Ozaki, T.~Ogita, S.~M. Rump, and S.~Oishi.
\newblock Fast algorithms for floating-point interval matrix multiplication.
\newblock {\em J. Computational Applied Mathematics}, 236(7):1795--1814, 2012.

\bibitem{Papageorgiou_Smyrlis_90}
D.~T. Papageorgiou and Y.~S. Smyrlis.
\newblock The route to chaos for the kuramoto-sivashinsky equation.
\newblock {\em Theoretical and Computational Fluid Dynamics}, 3(1):15--42,
  1991.

\bibitem{Robinson_book}
J.~C. Robinson.
\newblock {\em Infinite-dimensional dynamical systems}.
\newblock Cambridge Texts in Applied Mathematics. Cambridge University Press,
  Cambridge, 2001.
\newblock An introduction to dissipative parabolic PDEs and the theory of
  global attractors.

\bibitem{Rump_matrices_0}
S.~M. Rump.
\newblock Fast interval matrix multiplication.
\newblock {\em Numerical Algorithms}, 61(1):1--34, 2012.

\bibitem{Sivashinksy77}
G.~Sivashinsky.
\newblock {Nonlinear analysis of the hydrodynamic instability in laminar flames
  {I}. Derivation of basic equations}.
\newblock {\em Acta Astron.}, 4:1177--1206, 1977.

\bibitem{Papageorgiou_Smyrlis_91}
Y.~S. Smyrlis and D.~T. Papageorgiou.
\newblock {Predicting chaos for infinite dimensional dynamical systems: The
  Kuramoto-Sivashinsky equation, a case study }.
\newblock {\em {Proc. Natl. Acad. Sci}}, 88:11129--11132, 1991.

\bibitem{Temam97}
R.~Temam.
\newblock {\em Infinite-dimensional dynamical systems in mechanics and
  physics}, volume~68 of {\em Applied Mathematical Sciences}.
\newblock Springer-Verlag, New York, second edition, 1997.

\bibitem{TresserC78}
C.~Tresser and P.~Coullet.
\newblock It\'erations d'endomorphismes et groupe de renormalisation.
\newblock {\em C. R. Acad. Sci. Paris S\'er. A-B}, 287(7):A577--A580, 1978.

\bibitem{Piotr2}
P.~Zgliczynski.
\newblock {Attracting fixed points for the Kuramoto-Sivashinsky equation}.
\newblock {\em SIAM J. of Appl. Dyn. Syst.}, 1(2):215--235, 2002.

\bibitem{Piotr3}
P.~Zgliczynski.
\newblock {Rigorous numerics for dissipative Partial Differential Equations II.
  Periodic orbit for the Kuramoto-Sivashinsky PDE - a computer assisted proof}.
\newblock {\em {Foundations of Computational Mathematics}}, 4:157--185, 2004.

\end{thebibliography}
\bibliographystyle{abbrv}

\end{document}